\documentclass[reqno]{amsart}
\usepackage[english]{babel}
\usepackage{amscd,amssymb,amsmath,amsfonts,latexsym,amsthm}
\usepackage{inputenc}
\usepackage{graphicx,color}
\newtheorem{thm}{Theorem}[section]
\newtheorem{cor}[thm]{Corollary}
\newtheorem{lem}[thm]{Lemma}
\newtheorem{prop}[thm]{Proposition}
\newtheorem{defn}[thm]{Definition}

\numberwithin{equation}{section}
\DeclareMathOperator{\Der}{Der}
\DeclareMathOperator{\Hom}{Hom}

\begin{document}

\title[Pro-solvable Lie algebras]{On pro-solvable Lie algebras with maximal pro-nilpotent ideals $\mathfrak{m}_0$ and $\mathfrak{m}_2$.}

\author{K.K. Abdurasulov$^1$, B.A. Omirov$^{1,2}$,  G.O. Solijanova$^2$}

\address{$^1$ Institute of Mathematics Uzbekistan Academy of Sciences, Tashkent 100170, Uzbekistan.} \email{abdurasulov0505@mail.ru, omirovb@mail.ru}

\address{$^2$ National University of Uzbekistan, Tashkent 100174, Uzbekistan} \email{gulhayo.solijonova@mail.ru}

\begin{abstract} The paper is devoted to the study of pro-solvable Lie algebras whose maximal pro-nilpotent ideal is either $\mathfrak{m}_0$ or $\mathfrak{m}_2$. Namely, we describe such Lie algebras and establish their completeness. Triviality of the second cohomology group for one of the obtained algebra is established.
\end{abstract}



\maketitle
\normalsize

\section{Introduction}

Pro-solvable and pro-nilpotent Lie algebras are an important and interesting class of Lie algebras,
which generalize the class of solvable and nilpotent Lie algebras, respectively.

The definition of pro-solvable (respectively, pro-nilpotent) Lie algebra $L$ is divided into two parts; the first part of the definition is the condition of potentially solvability (respectively, nilpotency) and the second part of the definition is $\dim L^{[i]}/L^{[i+1]}<\infty$  (respectively, $\dim L^i/L^{i+1}<\infty)$ for any $i\geq 1$.

Since the study of infinite-dimensional solvable and pro-nilpotent Lie algebras is a complex problem, they should be studied by adding additional restrictions. One of such important restrictions for the study of solvable Lie algebras is fixing its nilradical, while for nilpotent Lie algebras one of successful restrictions is condition on
dimensions of $L^i/L^{i+1}.$ Here we apply these approaches for the study of pro-solvable Lie algebras by fixing their maximal pro-nilpotent ideals.

Due to works \cite{Yu} and \cite{Millio} we have some examples of pro-nilpotent Lie algebras. Among pro-nilpotent Lie algebras we consider those which have the most simple structure, they are $\mathfrak{m}_0, \ \mathfrak{m}_1$ and $\mathfrak{m}_2$. It is known that unique pro-solvable Lie algebra with maximal pro-solvable Lie algebra $\mathfrak{m}_1=\{e_i\ | \ i\in \mathbb{N}\}$ is algebra $W_{\geq 0}=\{e_i \ | \ i\in \mathbb{N}\cup \{0\}\}.$

Note that for the algebras $\mathfrak{m}_0, \ \mathfrak{m}_1$ and $\mathfrak{m}_2$ the conditions
$\dim(L^1/L^{2})=2, \ \dim(L^i/L^{i+1})=1\ i\geq1$ hold true.

Similar to finite-dimensional solvable Lie algebras considered in \cite{Qobil} we focus our study for pro-solvable Lie algebras with maximal pro-nilpotent ideals and maximal dimension of complementary subspace to the ideals.
In this work we describe pro-solvable 
Lie algebras generated by $\mathfrak{m}_0$ (respectively, $\mathfrak{m}_2$) and its special kinds of derivations under the condition that complementary subspace to $\mathfrak{m}_0$ has maximal dimension. We also prove that such algebras are complete. Furthermore, the triviality of the second cohomology group of one of them is proved.

Throughout the paper we consider complex Lie algebras with countable basis such that any element of the algebra can be represented as a finite linear combination of basis elements. Moreover, by maximal ideal we shall assume maximal by including ideal.

\normalsize
\section{Preliminaries}

In this section we give necessary definitions and preliminary results.

\begin{defn} An algebra $( L,[\cdot,\cdot])$  is called a
Lie algebra if it satisfies the properties
\begin{center}
\([x,x]=0\),\\
\([x,[y,z]]+[y,[z,x]] + [z,[x,y]]=0\)
\end{center}
for all \(x,y ,z\in L\). The second condition is called Jacobi identity.
\end{defn}

\begin{defn} A linear map $d \colon L \rightarrow L$ of an algebra $(L, [-,-])$ is said to be a
 \emph{derivation} if for all $x, y \in L$, the following  condition holds: \[d([x,y])=[d(x),y] + [x, d(y)] \,.\]
\end{defn}
For a given $x \in L$, $ad_x$ denotes the map $ad_x: L \rightarrow L$ such that $ad_x(y)=[x,y], \ \forall y \in L$.
One can check that a map $ad_x$ is a derivation. We call this kind of derivations  {\it inner derivations}.

\begin{defn} A Lie algebra $L$ is called  \emph{complete} if ${\rm Center}(L)=0$ and all derivations of $L$ are
inner \cite{Jacobson}.
\end{defn}
For a Lie algebra $L$ we define the {\it lower central} and the {\it derived series} as follow
$$L^1=L, \ L^{k+1}=[L^k,L],  \ k \geq 1, \qquad L^{[1]}=L, \ L^{[s+1]}=[L^{[s]},L^{[s]}], \ s \geq 1,$$
respectively.

\begin{defn} \cite{Millio} A Lie algebra $L$ is said to be pro-solvable (respectively, pro-nilpotent) if
$\bigcap_{i=1} ^{\infty}L^{[i]}=0$ and $L/L^{[i]}<\infty$ (respectively,
$\bigcap_{i=1} ^{\infty}L^{i}=0$ and $L/L^{i}<\infty$) for any $i\geq1$.
\end{defn}

\begin{defn} A linear map $\rho: L \to L$ is called potentially nilpotent, if $\cap_{i=1}^{\infty}(Im\ \rho^i)=0$ holds.
\end{defn}

Below we introduce the analogue of notion of nil-independency which plays a crucial role in the description of finite-dimensional solvable Lie algebras \cite{Mub}.

\begin{defn} Derivations $d_{1},d_{2},\dots,d_{n}$ of a Lie algebra $L$ over a field $\mathbb{F}$ are said to be potentially nil-independent, if a map $f=\alpha_{1}d_{1}+\alpha_{2}d_{2}+\ldots+\alpha_{n}d_{n}$ is not potentially nilpotent for any scalars $\alpha_{1},\alpha_{2},\dots,\alpha_{n}\in \mathbb{F}$. In other words,
$\cap_{i=1}^{\infty}Imf^i=0$ if and only if $\alpha_{1}=\alpha_{2}=\dots=\alpha_{n}=0.$
\end{defn}



For the definition of cohomology group of Lie algebras we refer readers to \cite{Jac}, \cite{Kac}. In this paper we just recall that
$${\rm H}^1(L,L)={\rm Der}(L)/{\rm Inder}(L) \quad \mbox{and} \quad {\rm H}^2(L,L)={\rm Z}^2(L,L)/{\rm B}^2(L,L)$$ where the set ${\rm Z}^2(L,L)$ consists of those elements $\varphi\in {\rm Hom}(\wedge^2L, L)$ such that
\begin{equation}\label{eq5}
[x,\varphi(y,z)] - [\varphi(x,y), z]
+[\varphi(x,z), y] +\varphi(x,[y,z]) - \varphi([x,y],z)+\varphi([x,z],y)=0,
\end{equation}
while ${\rm B}^2(L,L)$ consists of elements $\psi\in {\rm Hom}(\wedge^2L, L)$ such that
\begin{equation}\label{eq4}
\psi(x,y)=[d(x),y] + [x,d(y)] - d([x,y]) \,\, \mbox{for some linear map} \,\, d\in {\rm Hom}(L,L).
\end{equation}

In terms of cohomology groups the notion of completeness of a Lie algebra $L$ means that it is centerless and ${\rm H}^1(L,L)=0$.

For the convenience we introduce denotation
$$Z(a,b,c)=[a,\varphi(b,c)]-[\varphi(a,b),c]+[\varphi(a,c),b]+\varphi(a,[b,c])-\varphi([a,b],c)+\varphi([a,c],b).$$

\section{Main Part}

In this section we present the main results on pro-solvable Lie algebras whose maximal pro-nilpotent ideal is either $\mathfrak{m}_0$ or $\mathfrak{m}_2$ under the condition of maximality of complemented space to the ideals. Similar to the finite-dimensional case, first we give description of derivations of the ideals and then by using their properties we describe pro-solvable algebra. Finally, we are going to prove some properties of low order cohomologies of the obtained pro-solvable Lie algebras.

Let us consider the following infinite-dimensional Lie algebras
$$\begin{array}{ll}
\mathfrak{m}_0:\left\{\begin{array}{lllll}
[e_{i},e_{1}]=e_{i+1},& i\geq 2,
\end{array}\right.\quad \text{and}
&\quad \mathfrak{m}_2:\left\{\begin{array}{lllll}
[e_1, e_i] = e_{i+1},& i \geq 2,\\[1mm]
[e_2, e_j ] = e_{j+2}, & j \geq3.
\end{array}\right.
\end{array}$$

\begin{prop}\label{dif0} The spaces of derivations of the algebras $\mathfrak{m}_0$ and $\mathfrak{m}_2$ are the following

\normalsize
 $$\Der(\mathfrak{m}_0):\left\{\begin{array}{ll}d(e_{1})=\sum\limits_{i=1}^{t}\alpha_{i}e_{i},\\[2mm]
  d(e_{k})=((k-2)\alpha_{1}+\beta_{2})e_{k}+\sum\limits_{i=3}^{t}\beta_{i}e_{i+k-2},\,\,\ \text{where} \,\,k\geq 2,
  \end{array}\right.$$
  $$\Der(\mathfrak{m}_2):\left\{\begin{array}{ll}
  d(e_1)=\alpha_1e_1+\sum\limits_{i=3}^{n}\alpha_ie_i,\quad
d(e_2)=2\alpha_1e_2+\sum\limits_{i=3}^{n}\beta_ie_i,\\[2mm]
d(e_i)=i\alpha_1e_{i}+\beta_3e_{i+1}+\sum\limits_{j=i+2}^{n+i-2}
(\beta_{j-i+2}-\alpha_{j-i+1})e_j-\alpha_ne_{n+i-1},\, \,i\geq3.
\end{array}\right.$$
 \end{prop}
\begin{proof} We describe $\Der(\mathfrak{m}_0)$ and omit the description $\text{Der}(\mathfrak{m}_2)$ because it is similar. We set
$$d(e_{1})=\sum_{s=1}^{p_{1}}\alpha_{i_{s}}e_{i_{s}}, \quad d(e_{2})=\sum_{t=1}^{q_{2}}\beta_{j_{t}}e_{j_{t}}.$$
Without loss of generality one can assume  $$d(e_{1})=\sum_{i=1}^{t}\alpha_{i}e_{i},\quad d(e_{2})=\sum_{j=1}^{t}\beta_{j}e_{j}, \quad  max\{p,q\}=t.$$

Now by straightforward checking the derivation property and using the table
of multiplications of the algebras $\mathfrak{m}_0$, we obtain
\begin{equation}
 d(e_{1})=\sum_{i=1}^{t}\alpha_{i}e_{i}, \quad d(e_{k})=((k-2)\alpha_{1}+\beta_{2})e_{k}+\sum_{i=3}^{t}\beta_{i}e_{i+k-2}, \,k\geq 2.
  \end{equation}
\end{proof}

We denote by ${M}_0, M_2$ pro-solvable Lie algebras with maximal pro-nilpotent ideals are $\mathfrak{m}_0, \mathfrak{m}_2$, respectively and by $Q_1, Q_2$ complementary subspaces to $\mathfrak{m}_0, \mathfrak{m}_2$, respectively.

\begin{lem}\label{lem1} $ad_{x}$ and $ad_y$ are non-potentially nilpotent for any $x\in Q_1$ and $y\in Q_2$, respectively.
\end{lem}
\begin{proof} Let us assume the contrary and let $x\in Q_1$ such that $\cap_{k=1}^{\infty}Im\empty\ ad_{x}^{k}=0$. Set $\mathfrak{m}_0'=\mathfrak{m}_0+\mathbb{F}x$. Since $ad_{{x}|\mathfrak{m}_0}=d$ for some $d\in Der(\mathfrak{m}_0)$, from Proposition \ref{dif0} the condition $\cap_{k=1}^{\infty}Im\empty\ ad_{x}^{k}=0$ implies that $\alpha_{1}=\beta_{2}=0$. Therefore, we have
 $$[e_{1},x]=\sum_{i=2}^{t}\alpha_{i}e_{i}, \quad [e_{k},x]=\sum_{i=3}^{t}\beta_{i}e_{i+k-2}, \ \ k\geq 2.$$
One can easily check that $\cap_{i=1}^{\infty}(\mathfrak{m}_0')^{i}=0$, hence $\mathfrak{m}_0'$ is pro-nilpotent, which is a contradiction to maximality of $\mathfrak{m}_0.$ The rest of the assertion of the lemma is proved similarly.
\end{proof}

\begin{prop}\label{thm1} The dimensions of $Q_1$ and $Q_2$ are not greater than the maximal number of potentially nil-independent derivations of $\mathfrak{m}_0$ and $\mathfrak{m}_2$, respectively.
\end{prop}
\begin{proof} Let $\{{x_{1}, x_{2},\dots ,x_{m}}\}$ be a basis of $Q_1$. If $\sum_{i=1}^{m}\alpha_{1}ad_{x_{i}}=ad_{\sum\limits_{i=1}^{m}\alpha_{1}{x_{i}}}$ is a potentially nilpotent derivation of $\mathfrak{m}_0$, then by applying Lemma \ref{lem1} we obtain $\sum\limits_{i=1}^{m}\alpha_{1}{x_{i}}=0$, which imply $\alpha_{i}=0, i=1,\dots , m.$ Therefore,
the operators $ad_{x_{1}},\dots ,ad_{x_m}$ are potentially nil-independent.
Similarly, the rest assertion of proposition can be obtained.
\end{proof}

\begin{cor}
\label{cor1} The maximal number of potentially nil-independent derivations of $\mathfrak{m}_0$ and $\mathfrak{m}_2$ are not greater than $2$ and $1$, respectively.
\end{cor}

\begin{proof}
Suppose, the maximal number of potentially nil-independent derivations of $\mathfrak{m}_0$ is more than 2. Then there exist potentially nil-independent derivations $d_{1}, d_2, d_3$ of $\mathfrak{m}_0$.

We set
 $$d_s(e_{1})=\sum\limits_{i=1}^{t}\alpha_{i}^se_{i}, \quad d_s(e_{k})=((k-2)\alpha_{1}^s+\beta_{2}^s)e_{k}+\sum\limits_{i=3}^{t}\beta_{i}^se_{i+k-2}, \quad s=1, 2, 3.$$

Since for any values $\alpha_{i}^s,\beta_{2}^s$
the vectors $(\alpha_{1}^s, \beta_{2}^s), s=1, 2,3$ are linearly dependent, we conclude that $d_1, d_2$ and $d_3$ are potentially dependent derivations. However, we can choose values $\alpha_{i}^s,\beta_{2}^s, s=1,2$ such that the vectors $(\alpha_{i}^s, \beta_{2}^s), s=1, 2$ are linearly independent and they define nil-independent derivation.
\end{proof}

\begin{thm}\label{thm3} Let ${M}_0$ be a pro-solvable Lie algebra with maximal pro-nilpotent ideal $\mathfrak{m}_0$. Then it admits a basis $\{x, y, e_{1}, e_{2}, \dots\}$ such that the multiplication table of ${M}_0$ in this basis has the following form
$${M}_0(\beta):\left\{\begin{array}{ll}
[e_{i} ,e_{1}]=e_{i+1},\quad i\geq2,\\[1mm]
[x,e_1]=-e_1,\\[1mm]
[x,e_i]=(1-i)e_i-\sum\limits_{k=3}^{t}\beta_{k}e_{t+i-2},\\[1mm]
[y,e_i]=-e_i,\quad i\geq 2,
 \end{array}\right.$$
 where $\beta=(\beta_3,\beta_4,\dots,\beta_{t}).$
\end{thm}
\begin{proof} According to Proposition \ref{thm1} and Corollary \ref{cor1} we get the existence of a basis $\{x,y,e_{1},e_{2},\dots\}$ of ${M}_0$ such that $\mathfrak{m}_0=\{e_{1},e_{2},\dots\}$ and $Q_1=\{x,y\}$. By using derivation property we derive
\[[x,e_{1}]=-e_{1}-\sum\limits_{i=2}^{t}\alpha_{i}e_{i},\quad [y,e_{1}]=-\sum\limits_{i=2}^{t}\alpha_{i}e_{i},\quad
[x,y]=\sum\limits_{i=1}^{t}\gamma_ie_i,\]
\[[y,e_{k}]=-e_{k}-\sum\limits_{i=k+1}^{t+k-2}\beta_{i-k+2}'e_{i},  \quad [x,e_{k}]=(2-k)e_{k}+\sum\limits_{i=k+1}^{t+k-2}\beta_{i-k+2}e_{i},\quad k\geq2.\]

Taking $e_1'=e_1+\alpha_2e_2$ one can assume $\alpha_2=0$.

Now, setting
$$x'=x+\sum\limits_{i=2}^{t-1}\alpha_{i+1}e_i,\,\,\, y'=y+\gamma_1e_1+\sum\limits_{i=2}^{t-1}\alpha_{i+1}'e_{i}$$
we can assume
$$[x,e_1]=-e_1, \quad  [y,e_1]=-\alpha_2'e_2.$$

The equalities $[e_1,[x,y]]+[x,[y,e_1]]+[y,[e_1,x]]=0$ and $[e_2,[x,y]]+[x,[y,e_2]]+[y,[e_2,x]]=0$ imply $\alpha_2=0,\ \  \gamma_i=0, \  2\le i\le t$ and
$\beta'_i=0, \ 3\le i\le t$.

Finally, putting $x'=x+y$ we obtain the multiplication table in the assertion of the theorem.
\end{proof}

\begin{thm}\label{thm3} Let $M_2$ be pro-solvable Lie algebra with maximal pro-nilpotent ideal $\mathfrak{m}_2$. Then it is admits a basis $\{x, e_{1}, e_{2}, \dots\}$ such that the multiplication table of ${M}_2$ in this basis has the following form
$$M_2(\gamma):\left\{\begin{array}{ll}
[e_{1} ,e_{i}]=e_{i+1}, &i\geq 2, \\[1mm]
[e_2,e_j]=e_{j+2},&j\geq 3,\\[1mm]
[x,e_1]=-e_1,\\[1mm]
[x,e_i]=-ie_i-\sum\limits_{j=4}^{n+1}\gamma_je_{j+i-2},& i\geq 2,
\end{array}\right.$$
where $\gamma=(\gamma_4,\gamma_5,\ldots,\gamma_{n+1}).$
\end{thm}
\begin{proof} Due to Proposition \ref{thm1} and Corollary \ref{cor1} we have the existence of a basis $\{x,e_{1},e_{2},\dots\}$ of $M_2$ with $\mathfrak{m}_2=\{e_{1},e_{2},\dots\}$ and $Q_2=\{x\}$.
By using derivation property we derive
$$\begin{array}{llll}
[x,e_1]=-\alpha_1e_1-\sum\limits_{i=3}^{n}\alpha_ie_i,\quad\quad
[x,e_2]=-2\alpha_1e_2-\sum\limits_{i=3}^{n}\beta_ie_i,\\[1mm]
[x,e_i]=-i\alpha_1e_{i}-\beta_3e_{i+1}+\sum\limits_{j=i+2}^{n+i-2}
(\alpha_{j-i+1}-\beta_{j-i+2})e_j+\alpha_ne_{n+i-1},&i\geq3.
\end{array}$$

Since $\alpha_1\neq 0$ (because $ad_x$ is not potentially nilpotent), by scaling $x'=\frac{x_1}{\alpha_1}$ we can assume $\alpha_1=1$.

Making the change of basis element $x$ as follows
$$x'=x+\beta_3e_1-\sum\limits_{i=2}^{n-1}\alpha_{i+1}e_i,$$
we obtain
$$[x',e_1]=-e_1,$$
$$[x',e_2]=-2e_2-\sum\limits_{i=3}^{n}\beta_ie_i+\beta_3e_3+\sum\limits_{i=3}^{n-1}\alpha_{i+1}e_{i+2}
=-2e_2-\beta_4e_4+\sum\limits_{i=5}^{n}(\alpha_{i-1}-\beta_i)e_i+\alpha_ne_{n+1},$$
$$[x',e_i]=-ie_i-\beta_4e_{i+2}-\sum\limits_{j=5}^{n}(\beta_{j}-\alpha_{j-1})e_{j+i-2}+\alpha_ne_{n+i-1},\ \ i\geq 3.$$
Denoting $\gamma_4:=\beta_4$ and $\gamma_j=\beta_{j}-\alpha_{j-1},\  5\le j\le n$ we get the family of algebras $M_2(\gamma)$.
\end{proof}

Let us  consider the following pro-solvable Lie algebras
$$\begin{array}{lll}
\widetilde{{M}_0}:\left\{\begin{array}{llll}
[e_{i} ,e_{1}]=e_{i+1}, & i\geq 2, &\\[1mm]
[x,e_1]=-e_1,\\[1mm]
[x,e_i]=-(i-1)e_i,& i\geq 2,&\\[1mm]
[y,e_i]=-e_i,& i\geq 2,&
 \end{array}\right.& \widetilde{M_2}:\left\{\begin{array}{lll}
[e_{1} ,e_{i}]=e_{i+1}, & i\geq 2,& \\[1mm]
[e_2,e_j]=e_{j+2},& j\geq 3,&\\[1mm]
[x,e_1]=-e_1,\\[1mm]
[x,e_i]=-ie_i,& i\geq 2.&
 \end{array}\right.
\end{array}$$

\begin{prop}\label{prop3} The spaces of derivations of the algebras $\widetilde{{M}_0}$ and $\widetilde{M_2}$ are the following
$$\Der(\widetilde{{M}_0}): \quad \left\{\begin{array}{lll}
d(e_1)=\alpha_1e_1+\sum\limits_{i=3}^{n}\alpha_ie_i,&
d(x)=-\beta_3e_1+\sum\limits_{i=2}^{n-1}(i-1)\alpha_{i+1}e_i,\\[2mm]
d(e_i)=((i-2)\alpha_1+\beta_2)e_i+\beta_{3}e_{i+1},\quad i\geq2,&
d(y)=\sum\limits_{i=2}^{n-1}\alpha_{i+1}e_i.
\end{array}\right.$$
$$\Der(\widetilde{M_2}): \quad \left\{\begin{array}{llll}
d(e_1)=\alpha_1e_1+\sum\limits_{i=3}^{n-1}\alpha_ie_i,& & d(e_2)=2\alpha_1e_2+\beta_3e_3+\sum\limits_{i=4}^{n}\alpha_{i-1}e_i,\\[1mm]
d(e_i)=i\alpha_1e_i+\beta_3e_{i+1}-\alpha_3e_{i+2},& i\geq3, &d(x)=\beta_3e_1-\sum\limits_{i=2}^{n-1}i\alpha_{i+1}e_i.\\[1mm]
\end{array}\right.
$$

\end{prop}
\begin{proof}

We give the description of $Der(\widetilde{{M}_0})$ and the space of derivation for the algebra
$\widetilde{M_2}$ can be obtained by analogy.

Without loss of generality one can assume that $$\begin{array}{lll}
  d(e_1)=\sum\limits_{i=1}^{n}\alpha_{i}e_i+\alpha_{1,1} x+\alpha_{2,2}y,&
  d(e_2)=\sum\limits_{i=1}^{n}\beta_{i}e_i+\beta_{1,1}x+\beta_{2,2}y,\\[1mm]
  d(x)=\sum\limits_{i=1}^{n}\gamma_{i}e_i+\gamma_{1,1}x+\gamma_{2,2}y,&
  d(y)=\sum\limits_{i=1}^{n}\tau_{i}e_i+\tau_{1,1}x+\tau_{2,2}y,\end{array}$$
where  $d\in Der(\widetilde{{M}_0})$ and $n$ is a big enough natural number.

 By using the table of multiplication of the algebra ${M}_0$ and derivation property we obtain
 \[d(e_3)=d([e_2,e_1])=-\beta_{1,1}e_1+(\alpha_{1,1}+\alpha_{2,2})e_2+(\alpha_1+\beta_2)e_3+
 \sum\limits_{i=4}^{n+1}\beta_{i-1}e_i,\]
\[d(e_4)=d([e_3,e_1])=(3\alpha_{1,1}+2\alpha_{2,2})e_3+(2\alpha_{1}+\beta_2)e_4+\sum\limits_{i=5}^{n+2}\beta_{i-2}e_i,\]

From  $d([e_2,e_3])=d([e_3,e_4])=0$ and $d(e_1)=d([e_1,x])$ we derive

$$\alpha_{1,1}=\alpha_{2,2}=\beta_{1,1}=\beta_{2,2}=\gamma_{1,1}=\gamma_n=0,\quad \gamma_i=(i-1)\alpha_{i+1},\quad 2\le i\le n-1.$$

The equalities
$$[x,e_2]=-e_2, \quad [y,e_1]=0, \quad [y,e_2]=-e_2$$
imply
$$\beta_i=0,\quad 4\le i\le n,\quad \gamma_1=-\beta_3, \quad \alpha_2=\gamma_{2,2}=\tau_{1,1}=\tau_1=\tau_{2,2}=\tau_n=0,\quad \tau_i=\alpha_{i+1},\quad 2\le i\le n-1.$$

By induction we get $d(e_{i+1})=d([e_i,e_1])=((i-1)\alpha_1+\beta_2)e_{i+1}+\beta_{3}e_{i+2},\ \  i\geq3,$ which completes the description of the space $Der(\widetilde{{M}_0})$.
\end{proof}

In the next theorem we prove the completeness of the algebras $\widetilde{{M}_0}$ and $\widetilde{M_2}$.
\begin{thm}\label{thm5} The pro-solvable Lie algebras $\widetilde{{M}_0}$ and $\widetilde{M_2}$ are complete.
\end{thm}
\begin{proof}
It follows from the corresponding table of multiplications that  $\widetilde{{M}_0}$ and $\widetilde{{M}_2}$  have trivial center.

Let us consider derivation $d\in Der(\widetilde{{M}_0})$ in the form as it presented in Proposition \ref{prop3}. Then $d=ad_a$ for $a=-\beta_{3}e_1+\sum\limits_{k=2}^{n-1}\alpha_{k+1}e_{k}-\alpha_1x+(\alpha_1-\beta_2)y.$

Let now $d'\in Der(\widetilde{M_2})$ has the form as in previous proposition. Then $d'=ad_{b}$ for $b=\beta_{3}e_1-\sum\limits_{k=2}^{n-2}\alpha_{k+1}e_{k}-\alpha_1x$.
\end{proof}

Finally, we are going to study the second cohomology groups for the algebras $\widetilde{{M}_0}$ and $\widetilde{M_2}$ in adjoint representations.

\begin{prop}\label{prop6} An arbitrary element $\varphi\in {\rm Z}^2(\widetilde{{M}_0},\widetilde{{M}_0})$ has the following form
$$\left\{\begin{array}{lll}
\varphi(e_i,e_1)=&(\gamma_1^{i+1}+\gamma_{1,1}^i)e_1+\sum\limits_{k=2}^{i}\frac{1}{i-k+1}\beta_k^{i+1}e_k+
\sum\limits_{k=i+2}^{p_{i+1}}\frac{1}{i-k+1}\beta_k^{i+1}e_k
+\alpha_{i+1}^{i,1}e_{i+1}-\\[1mm]
&-\sum\limits_{k=3}^{i}\frac{1}{i-k+1}\beta_{k-1}^ie_k-\sum\limits_{k=i+2}^{p_i}
\frac{1}{i-k+1}\beta_{k-1}^ie_k-((i-1)\beta_{1,1}^1+\beta_{2,2}^1)e_i+\gamma_{1,1}^{i+1}x+\gamma_{2,2}^{i+1}y,&\\[2mm]
\varphi(e_i,e_j)=&-(\gamma_{1}^je_{i+1}+((i-1)\gamma_{1,1}^j+\gamma_{2,2}^j)e_i)+
(\gamma_{1}^ie_{j+1}+((j-1)\gamma_{1,1}^i+\gamma_{2,2}^i)e_j),&\\[2mm]
\varphi(e_1,x)=&\beta_{1}^1e_1+\sum\limits_{k=3}^{p_1}\beta_k^1e_k+\beta_{1,1}^1x+\beta_{2,2}^1y,\\[2mm]
\varphi(e_i,x)=&(i-2)\gamma_{1}^ie_1+\sum\limits_{k=2}^{i-1}\beta_k^ie_k+\sum\limits_{k=i+1}^{p_i}\beta_k^ie_k+(\beta_2^2+(i-2)\beta_1^1)e_i+
(i-1)(\gamma_{1,1}^ix+\gamma_{2,2}^iy),& \\[2mm]
\varphi(e_1,y)=&\sum\limits_{k=1}^{q_1}\gamma^{1}_{k}e_k,\\[2mm]
\varphi(e_i,y)=&\gamma_{1}^{i}e_1+(\gamma_{2}^2+(i-2)\gamma_{1}^1)e_{i}+\tau_1e_{i+1}+\gamma_{1,1}^ix+\gamma_{2,2}^iy,&\\[2mm]
\varphi(x,y)=&-\tau_1e_1+\sum\limits_{k=2}^{q_1-1}(k-1)\gamma^{1}_{k+1}e_k+\sum\limits_{k=1}^{p_1-1}\beta^{1}_{k+1}e_k
\end{array}\right.$$
where $i,j\geq 2$.
\end{prop}
\begin{proof} For $\varphi\in Z^2(\widetilde{{M}_0},\widetilde{{M}_0})$ we set
$$\begin{array}{ll}
\varphi(e_i,e_j)=\sum\limits_{k=1}^{p_{i,j}}\alpha^{i,j}_ke_k+\alpha^{i,j}_{1,1}x+\alpha^{i,j}_{2,2}y,\quad j>i\geq 1,&
\varphi(e_i,x)=\sum\limits_{k=1}^{p_{i}}\beta^{i}_ke_k+\beta^{i}_{1,1}x+\beta^{i}_{2,2}y,\quad i\geq 1,\\[2mm]
\varphi(e_i,y)=\sum\limits_{k=1}^{q_{i}}\gamma^{i}_ke_k+\gamma^{i}_{1,1}x+\gamma^{i}_{2,2}y,\quad i\geq 1,&
\varphi(x,y)=\sum\limits_{k=1}^{t}\tau_ke_k+\tau_{1,1}x+\tau_{2,2}y.\\[2mm]
\end{array}$$

Taking into account the table of multiplication of the algebra $\widetilde{{M}_0}$ in equality (\ref{eq5}) for the element $\varphi$ we derive the following
\begin{center}
    \begin{tabular}{lll}
        2-cocyle identity & &Constraints\\
        \hline\hline
        $Z(e_i,e_j,y)=0,\ i,j\geq2$ &\quad $\Rightarrow $\quad & $\left\{\begin{array}{ll}
                                                                 \alpha_{1,1}^{i,j}=\alpha_{2,2}^{i,j}=0,\ \alpha_{i}^{i,j}=-((i-1)\gamma_{1,1}^i+\gamma_{2,2}^i),\\[1mm]
                                                                 \alpha_{i+1}^{i,j}=-\gamma_{1}^j,\ \alpha_{j}^{i,j}=(j-1)\gamma_{1,1}^i+\gamma_{2,2}^i,\
                                                                 \alpha_{j+1}^{i,j}=\gamma_1^i,\\[1mm]
                                                                 \alpha_{k}^{i,j}=0,\quad \text{where} \ k\neq\{i,i+1,j,j+1\} \ \text{and}\quad 2\le k\le p_{i,j},
                                                                 \end{array}\right.$\\
        $Z(e_1,x,y)=0,$ &\quad $\Rightarrow $\quad & $\left\{\begin{array}{ll}
                                                                 \tau_{1,1}=0,\quad \gamma_{1,1}^1=\gamma_{2,2}^{1}=0,\quad \beta_{2}^1=0,\\[1mm]
                                                                 \sum\limits_{k=3}^{t+1}\tau_{k-1}e_k=-\sum\limits_{k=3}^{p_1}\beta_{k}^1e_k+
                                                                 \sum\limits_{k=3}^{q_1}(k-2)\gamma_{k}^1e_k,\\[1mm]
                                                                 \end{array}\right.$\\
        $Z(e_i,x,y)=0,\  i\geq2,$ &\quad $\Rightarrow $\quad & $\left\{\begin{array}{ll}
                                                                 \tau_{2,2}=0,\ \tau_1=-\gamma_{i+1}^i,\ \beta_{1}^i=(i-2)\gamma_1^i,\\[1mm]
                                                                 \beta_{1,1}^i=(i-1)\gamma_{1,1}^i,\ \beta_{2,2}^i=(i-1)\gamma_{2,2}^i,\\[1mm]
                                                                 \gamma_k^i=0,\ \text{where}\  i\geq 2\ \text{and} \ k\neq\{1,i,i+1\},\\[1mm]
                                                                 \end{array}\right.$\\
        $Z(e_i,e_1,x)=0,\ i\geq2,$ &\quad $\Rightarrow $\quad & $\left\{\begin{array}{ll}
                                                                 \alpha_1^{i,1}=\gamma_1^{i+1}+\gamma_{1,1}^i,\ \alpha_{1,1}^{i,1}=\gamma_{1,1}^{i+1},\
                                                                 \alpha_{2,2}^{i,1}=\gamma_{2,2}^{i+1},\\[1mm]
                                                                 \beta_i^i=\beta_{2}^2+(i-2)\beta_1^1,\quad 3\le i\le p_i,\\[1mm]
                                                                 \alpha_i^{i,1}=\beta_i^{i+1}-(i-1)\beta_{1,1}^1-\beta_{2,2}^1-\beta_{i-1}^i,\
                                                                 \alpha_2^{i,1}=\frac{1}{i-1}\beta_2^{i+1},\\[1mm]
                                                                 \alpha_k^{i,1}=\frac{1}{i-k+1}(\beta_k^{i+1}-\beta_{k-1}^i),\quad k\neq i+1.\\[1mm]
                                                                 \end{array}\right.$\\
    \end{tabular}
\end{center}
Summarizing the obtained above restrictions we complete the proof of theorem.
\end{proof}

\begin{thm}
 $$H^2(\widetilde{{M}_0},\widetilde{{M}_0})=0, \quad H^2(\widetilde{M_2},\widetilde{M_2})\neq0.$$
\end{thm}
\begin{proof} Let $\varphi\in Z^2(\widetilde{{M}_0},\widetilde{{M}_0})$ has the form as in Proposition \ref{prop3}. We consider the map $f\in \text{End}(\widetilde{M_0})$ defined as follows
$$\left\{\begin{array}{lll}
f(e_1)=\gamma_2^1e_2+\sum\limits_{k=3}^{p_i}\frac{1}{k-2}\beta_k^1e_k-\beta_{1,1}^1x-\beta_{2,2}^1y,\\[2mm]
f(e_i)=-\gamma_1^ie_1+\sum\limits_{k=2}^{i-1}\frac{1}{k-i}\beta_k^ie_k-\sum\limits_{k=3}^{i}\alpha_{k}^{k-1,1}e_i+\sum\limits_{k=i+1}^{p_i}\frac{1}{k-i}\beta_{k}^{i}e_k-\gamma_{1,1}^ix-\gamma_{2,2}^{i}y,&i\geq 2.\\[2mm]
f(x)=\beta_1^1e_1+\beta_{1}^{1}x+(\beta_2^2-\beta_1^1)y,\\[2mm]
f(y)=\tau_1e_1-\sum\limits_{k=2}^{q_1-1}\gamma_{k+1}^1e_k-\sum\limits_{k=2}^{p_1-1}\frac{1}{k-1}\beta_{k+1}^1e_k+
\gamma_1^1x+(\gamma_2^2-\gamma_1^1)y.\\[2mm]
\end{array}\right.$$

Then $(\varphi-d^1f)(a,b)=0$ for any $a,b \in \widetilde{{M}_0}.$ Thus, $H^2(\widetilde{{M}_0},\widetilde{{M}_0})=0.$

Consider the map $\varphi\in \Hom(\widetilde{M_2}\wedge \widetilde{M_2}, \widetilde{M_2})$ defined as follows
$$\left\{\begin{array}{ll}
\varphi(e_3,e_j)=e_{j+3},\\[1mm]
\varphi(e_2,e_j)=(4-j)e_{j+2},& j\geq5,\\[1mm]
\varphi(a,b)=0, & \mbox{for the rest} \ a,b \in \widetilde{M_2}.\\[1mm]
\end{array}\right.$$

It is easy to check that  $\varphi \in Z^2(\widetilde{M_2},\widetilde{M_2})$.

Let us assume that there exists $f\in End(\widetilde{M_2})$ such that $(d^1f-\varphi)(a,b)=0$ for any $a,b \in \widetilde{M_2}.$

Setting $f(e_k)=\sum\limits_{i=1}^{t_k}\alpha_i^ke_i+A^kx$, we obtain
$$(df^1-\varphi)(e_3,e_5)=[f(e_3),e_5]+[e_3,f(e_5)]-f([e_3,e_5])-e_{8}=$$
$$=\alpha_1^3e_{6}+\alpha_{2}^3e_{7}-5A^3e_5-\alpha_1^5e_{4}-\alpha_2^5e_5+3A^5e_3-e_{8}\neq 0,$$
which is a contradiction to the assumption $d^1f=\varphi$. Therefore, $\varphi \not \in B^2(\widetilde{M_2},\widetilde{M_2})$.
\end{proof}

\end{document}